\documentclass[11pt]{article}

\usepackage{amsmath, amssymb, amsthm}
\usepackage{mathrsfs}
\usepackage{enumerate}
\usepackage{graphicx}   
\usepackage{hyperref}
\usepackage{enumitem}
\usepackage{cite}       
\usepackage{geometry}
\newtheorem{theorem}{Theorem}[section]
\newtheorem{lemma}[theorem]{Lemma}

\newtheorem{corollary}[theorem]{Corollary}
\newtheorem{definition}[theorem]{Definition}
\newtheorem{remark}[theorem]{Remark}
\newtheorem{example}[theorem]{Example}

\title{ON COMMON FIXED POINTS OF WEAKLY COMPATIBLE  MAPPINGS SATISFYING `GENERALIZED CONDITION (B)'}
\author{MUJAHID ABBAS, G. V. R. BABU AND G. N. ALEMAYEHU}

\begin{document}

\maketitle

\begin{abstract}
We prove the existence of common fixed points for two weakly compatible mappings satisfying a 'generalized condition (B)'. This result generalizes some theorems of Al-Thagafi and Shahzad \cite{AlThagafi2006} and Babu, Sandhya and Kameswari \cite{Babu2008}.


\end{abstract}

\section{INTRODUCTION AND PRELIMINARIES}
In 1968, Kannan \cite{Kannan1968} proved a fixed point theorem for a map satisfying a contractive condition that did not require continuity at each point. This paper was a genesis for a multitude of fixed point papers over the next two decades. Sessa \cite{Sessa1982} coined the term weakly commuting maps. Jungck \cite{Jungck1986} generalized the notion of weak commutativity by introducing compatible maps and then weakly compatible maps \cite{Jungck1996}.

We now introduce almost contraction property to a pair of selfmaps as follows:

\begin{definition}
Let $(X,d)$ be a metric space. A map $T:X\to X$ is called an \textit{almost contraction} with respect to a mapping $f:X\to X$ if there exist a constant $\delta\in ]0,1[$ and some $L\geq 0$ such that
\[d(Tx,Ty)\leq \delta\, d(fx,fy)+L\, d(fy,Tx),\]
for all $x,y\in X.$ If we choose $f=I_{X}$, $I_{X}$ is the identity map on $X$, we obtain the definition of \textit{almost contraction}, the concept introduced by Berinde \cite{Berinde2004, Berinde2008}.
\end{definition}

This concept was introduced by Berinde as 'weak contraction' in \cite{Berinde2004}. But in \cite{Berinde2008}, Berinde renamed 'weak contraction' as 'almost contraction' which is appropriate.

Berinde \cite{Berinde2004} proved the following two fixed point theorems for \textit{almost contractions} in complete metric spaces.

\begin{theorem}\label{thm1.2}
Let $(X,d)$ be a complete metric space and $T:X\to X$ an almost contraction. Then
\begin{enumerate}
\item $F(T)=\{x\in X:Tx=x\}\neq\emptyset$,
\item for any $x_{0}\in X$, the Picard iteration
\[x_{n+1}=Tx_{n}\;,n=0,1,2,\cdots\] 
converges to some $x^{*}\in F(T)$,
\item the following estimates
\[d(x_{n},x^{*})\leq\frac{\delta^{n}}{1-\delta}d(x_{0},x_{1})\mbox{ and }d(x_{n},x^{*})\leq\frac{\delta}{1-\delta}d(x_{n-1},x_{n})\]
hold, for $n=1,2,\cdots.$
\end{enumerate}
\end{theorem}

\begin{theorem}\label{thm1.3}
Let $(X,d)$ be a complete metric space and $T:X\to X$ an almost contraction for which there exist $\theta\in ]0,1[$ and some $L_{1}\geq 0$ such that
\[d(Tx,Ty)\leq\theta d(x,y)+L_{1}d(x,Tx),\quad\mbox{ for all }x,y\in X.\]
Then
\begin{enumerate}
\item $T$ has a unique fixed point, i.e., $F(T)=\{x^{*}\}$,
\item for any $x_{0}\in X$, the Picard iteration $\{x_{n}\}_{n=0}^{\infty}$ defined by (1.1) converges to $x^{*}\in F(T)$,
\item a priori and posteriori error estimates
\[d(x_{n},x^{*})\leq\frac{\delta^{n}}{1-\delta}d(x_{0},x_{1})\mbox{ and }d(x_{n},x^{*})\leq\frac{\delta}{1-\delta}d(x_{n-1},x_{n})\]
hold, for $n=1,2,\cdots,$
\item the rate of convergence of Picard iteration $\{x_{n}\}_{n=0}^{\infty}$ defined by (1.1) is given by
\[d(x_{n},x^{*})\leq\theta d(x_{n-1},x^{*})\]
for $n=0,1,2,\cdots.$
\end{enumerate}
\end{theorem}

These results were originally proved by Berinde \cite{Berinde2004, Berinde2008} in the context of almost contractions.

It was shown in \cite{Berinde2004} that any strict contraction, the Kannan \cite{Kannan1968} and Zamfirescu \cite{Zamfirescu1972} mappings, as well as a large class of quasi-contractions, are all almost contractions.

Let $T$ and $f$ be two selfmaps of a metric space $(X,d)$. $T$ is said to be \textit{$f$-contraction} if there exists $k\in[0,1)$ such that $d(Tx,Ty)\leq kd(fx,fy)$ for all $x,y\in E$.

In 2006, Al-Thagafi and Shahzad \cite{AlThagafi2006} proved the following theorem which is a generalization of many known results.

\begin{theorem}[ (Al-Thagafi and Shahzad \cite{AlThagafi2006}, Theorem 2.1)]\label{thm1.4}
Let $E$ be a subset of a metric space $(X,d)$ and $f$ and $T$ be selfmaps of $E$ and $T(E)\subseteq f(E)$. Suppose that $f$ and $T$ are weakly compatible, $T$ is $f$-contraction and $T(E)$ is complete. Then $f$ and $T$ have a unique common fixed point in $E$.
\end{theorem}

Recently Babu, Sandhya and Kameswari \cite{Babu2008} considered the class of mappings that satisfy 'condition (B)'.

Let $(X,d)$ be a metric space. A map $T:X\to X$ is said to satisfy \textit{'condition (B)'} if there exist a constant $\delta\in]0,1[$ and some $L\geq 0$ such that
\[d(Tx,Ty)\leq\delta d(x,y)+L\min\{d(x,Tx),d(y,Ty),d(x,Ty),d(y,Tx)\},\]
for all $x,y\in X$.

They proved the following fixed point theorem.

\begin{theorem}[(Babu, Sandhya and Kameswari \cite{Babu2008}, Theorem 2.3)]\label{thm1.5}
Let $(X,d)$ be a complete metric space and $T:X\to X$ be a map satisfying condition (B). Then $T$ has a unique fixed point.
\end{theorem}

\begin{definition} \label{def1.6}
A pair $(f,T)$ of self-mappings on $X$ is said to be \textit{weakly compatible} if $f$ and $T$ commute at their coincidence point (i.e. $fTx=Tfx,\ x\in X$ whenever $fx=Tx$). A point $y\in X$ is called a \textit{point of coincidence} of two self-mappings $f$ and $T$ on $X$ if there exists a point $x\in X$ such that $y=Tx=fx$.
\end{definition}

The following lemma is Proposition 1.4 of \cite{Abbas2008}.

\begin{lemma} \label{lem1.7}
Let $X$ be a non-empty set and the mappings $T, f:X\to X$ have a unique point of coincidence $v$ in $X$. If the pair $(f,T)$ is weakly compatible, then $T$ and $f$ have a unique common fixed point.
\end{lemma}

\begin{definition} \label{def1.8}
Let $(X,d)$ be a metric space, $T$ and $f$ be self-mappings on $X$, with
\end{definition}

Let $T(X)\subset f(X)$, and $x_{0}\in X$. Choose a point $x_{1}$ in $X$ such that $fx_{1}=Tx_{0}$. This can be done since $T(X)\subset f(X)$. Continuing this process having chosen $x_{1},\cdots,x_{k}$, we choose $x_{k+1}$ in $X$ such that
\[fx_{k+1}=Tx_{k},\quad k=0,1,2,\cdots.\]
The sequence $\{fx_{n}\}$ is called a \textit{$T$-sequence} with initial point $x_{0}$.

We now introduce a generalization of 'condition $(B)$' for a pair of self maps.

\begin{definition} \label{def1.9}
A selfmap $T$ on a metric space $X$ is said to satisfy \textit{'generalized condition (B)'} associated with a selfmap $f$ of $X$ if there exists $\delta\in ]0,1[$ and $L\geq 0$ such that
\[d(Tx,Ty)\leq\delta M(x,y)+L\min\{d(fx,Tx),d(fy,Ty),d(fx,Ty),d(fy,Tx)\}\tag{1.2}\]
for all $x,y\in X$, where
\[M(x,y)=\max\left\{d(fx,fy),d(fx,Tx),d(fy,Ty),\frac{d(fx,Ty)+d(fy,Tx)}{2}\right\}.\]
If $f=I_{X}$, then we say that $T$ satisfies 'generalized condition $(B)$'.
\end{definition}

Here we observe that 'condition $(B)$' implies 'generalized condition $(B)$'. But its converse need not be true.

\begin{example} \label{ex1.10}
Let $X=\{0,\frac{1}{2},1\}$ with the usual metric. We define a mapping $T:X\to X$ by
\[Tx=\begin{cases}
\frac{1}{2} & \text{if } x\in\{0,\frac{1}{2}\},\\
0 & \text{if } x=1.
\end{cases}\]
Then $T$ satisfies generalized condition $(B)$ with $\delta=\frac{1}{2}$ and $L=0$. But $T$ does not satisfy condition $(B)$, for by taking $x=\frac{1}{2}$ and $y=1$; condition $(B)$ fails to hold for any $\delta\in ]0,1[$ and any $L\geq 0$.
\end{example}

Recently, Berinde established the following fixed point result.

\begin{theorem}[(Berinde \cite{Berinde2008}, Theorem 3.4)]\label{thm1.11}
Let $(X,d)$ be a complete metric space and $T:X\to X$ a mapping for which there exist $\alpha\in ]0,1[$ and some $L\geq 0$ such that for all $x,y\in X$
\[d(Tx,Ty)\leq\alpha M(x,y)+L\min\{d(x,Tx),d(y,Ty),d(x,Ty),d(y,Tx)\},\]
where
\end{theorem}
\[M(x,y)=\max\{d(x,y),d(x,Tx),d(y,Ty),d(x,Ty),d(y,Tx)\}.\]
Then
\begin{enumerate}
\item[(1)] $T$ has a unique fixed point, i.e., $F(T)=\{x^{*}\}$;

\item[(3)] for any $x_{0}\in X$, the Picard iteration $\{x_{n}\}_{n=0}^{\infty}$ defined by (1.1) converges to some $x^{*}\in F(T)$;

\item[(3)] the priori estimate
\[d(x_{n},x^{*})\leq\frac{\alpha^{n}}{(1-\alpha)^{2}}d(x_{0},x_{1})\]
holds, for $n=1,2,\cdots$;

\item[(4)] the rate of convergence of Picard iteration is given by
\[d(x_{n},x^{*})\leq\theta\ d(x_{n-1},x^{*})\]
for $n=0,1,2,\cdots$.
\end{enumerate}

In this paper, we prove a result on the existence of points of coincidence for two maps satisfying generalized condition $(B)$. We apply this result to obtain common fixed points of two weakly compatible selfmaps, which is the main result of this paper (Theorem 2.2). Our result generalizes some theorems of Al-Thagafi and Shahzad \cite{AlThagafi2006} and Babu, Sandhya and Kameswari \cite{Babu2008}.

\section{COMMON FIXED POINT THEOREMS}

First, we establish a result on the existence of points of coincidence and then we apply this result to obtain common fixed points for two self mappings of weakly compatible maps.

\begin{theorem}\label{thm2.1}
Let $(X,d)$ be a metric space. Let $T,f:X\to X$ be such that $T(X)\subseteq f(X)$. Assume that $T$ satisfies generalized condition $(B)$ associated with $f$. If either $f(X)$ or $T(X)$ is a complete subspace of $X$, then $T$ and $f$ have a unique point of coincidence.
\end{theorem}

\begin{proof}
Let $x_{0}$ be an arbitrary point in $X$ and let $\{fx_{n}\}$ be a $T-$sequence with

initial point $x_0$. Now,
\begin{align*}
M(x_n, x_{n-1}) 
&= \max\left\{d(fx_n, fx_{n-1}), d(fx_n, Tx_n), d(fx_{n-1}, Tx_{n-1}), \right.\\
&\quad \left. \frac{d(fx_n, Tx_{n-1}) + d(fx_{n-1}, Tx_n)}{2}\right\}\\
&= \max\left\{d(fx_n, fx_{n-1}), d(fx_n, fx_{n+1}), d(fx_{n-1}, fx_n), \right.\\
&\quad \left. \frac{d(fx_n, fx_n) + d(fx_{n-1}, fx_{n+1})}{2}\right\}\\
&= \max\left\{d(fx_n, fx_{n-1}), d(fx_n, fx_{n+1}), \frac{d(fx_{n-1}, fx_{n+1})}{2}\right\}.
\end{align*}

Thus by taking $x_n$ for $x$ and $x_{n-1}$ for $y$ in the inequality (1.2), it follows that
\begin{align*}
d(Tx_n, Tx_{n-1}) 
&\leq \delta \max\left\{d(fx_n, fx_{n-1}), d(fx_n, fx_{n+1}), \frac{d(fx_{n-1}, fx_{n+1})}{2}\right\} \\
&\quad + L \min\left\{d(fx_n, fx_{n+1}), d(fx_{n-1}, fx_n), d(fx_n, fx_n), d(fx_{n-1}, fx_{n+1})\right\}
\end{align*}
which further gives that
\[
d(fx_n, fx_{n+1}) \leq \delta \max\left\{d(fx_n, fx_{n-1}), d(fx_n, fx_{n+1}), \frac{d(fx_{n-1}, fx_{n+1})}{2}\right\}.
\]

Now if $\max\left\{d(fx_n, fx_{n-1}), d(fx_n, fx_{n+1}), \frac{d(fx_{n-1}, fx_{n+1})}{2}\right\} = d(fx_n, fx_{n-1})$, then
\[
d(fx_n, fx_{n+1}) \leq \delta d(fx_{n-1}, fx_n).
\]

If $\max\left\{d(fx_n, fx_{n-1}), d(fx_n, fx_{n+1}), \frac{d(fx_{n-1}, fx_{n+1})}{2}\right\} = d(fx_n, fx_{n+1})$, then
\[
d(fx_n, fx_{n+1}) \leq \delta d(fx_n, fx_{n+1})
\]
which implies that $d(fx_n, fx_{n+1}) = 0$ and hence $fx_n = fx_{n+1} = Tx_n$ and the result follows.

Finally, $\max\left\{d(fx_n, fx_{n-1}), d(fx_n, fx_{n+1}), \frac{d(fx_{n-1}, fx_{n+1})}{2}\right\} = \frac{d(fx_{n-1}, fx_{n+1})}{2}$ gives that
\begin{align*}
d(fx_n, fx_{n+1}) 
&\leq \frac{\delta}{2}d(fx_{n-1}, fx_{n+1}) \\
&\leq \frac{\delta}{2}d(fx_{n-1}, fx_n) + \frac{\delta}{2}d(fx_n, fx_{n+1})
\end{align*}
which implies that
\[
d(fx_n, fx_{n+1}) \leq \frac{\delta}{2-\delta}d(fx_{n-1}, fx_n).
\]

\begin{align*}
d(fx_{n},fx_{n+1}) &\leq \delta d(fx_{n-1},fx_{n}) \\
&\leq \cdots \leq \delta^{n}d(fx_{0},fx_{1}).
\end{align*}

Now, for any positive integers $m$ and $n$ with $m>n$, we have
\begin{align*}
d(fx_{m},fx_{n}) &\leq d(fx_{n},fx_{n+1})+d(fx_{n+1},fx_{n+2})+\cdots+d(fx_{m-1},fx_{m}) \\
&\leq [\delta^{n}+\delta^{n+1}+\cdots+\delta^{m-1}]d(fx_{0},fx_{1}) \\
&\leq \frac{\delta^{n}}{1-\delta}d(fx_{0},fx_{1}),
\end{align*}
which implies that $\{fx_{n}\}$ is a Cauchy sequence. If $f(X)$ is a complete subspace of $X$, there exists a $p\in f(X)$ such that $fx_{n}\to p$. Hence we can find $u^{*}$ in $X$ such that $fu^{*}=p$. Now,

\begin{align*}
d(p,Tu^{*}) &\leq d(p,fx_{n+1})+d(fx_{n+1},Tu^{*}) \\
&= d(p,fx_{n+1})+d(Tx_{n},Tu^{*}) \\
&\leq d(p,fx_{n+1}) + \delta\max\left\{d(fx_{n},fu^{*}),d(fx_{n},Tx_{n}),d(fu^{*},Tu^{*}),\right. \\
&\quad \left.\frac{d(fx_{n},Tu^{*})+d(fu^{*},Tx_{n})}{2}\right\} \\
&\quad + L\min\left\{d(fx_{n},fx_{n+1}),d(fu^{*},Tu^{*}),d(fx_{n},Tu^{*}),d(fu^{*},fx_{n+1})\right\}
\end{align*}

which on taking limit as $n\to\infty$ gives that
\begin{align*}
d(p,Tu^{*}) &\leq \delta\max\left\{d(p,p),d(p,p),d(p,Tu^{*}),\frac{d(p,Tu^{*})+d(p,p)}{2}\right\} \\
&\quad + L\min\left\{d(p,p),d(p,Tu^{*}),d(p,Tu^{*}),d(p,p)\right\}
\end{align*}

which further implies
\[
d(p,Tu^{*})\leq\delta d(p,Tu^{*}).
\]

Hence $d(p,Tu^{*})=0$ and $fu^{*}=p=Tu^{*}$.

Now, if $T(X)$ is complete, then there exists a $q\in T(X)$ such that $Tx_{n}\to q$ as $n\to\infty$. Since $T(X)\subset f(X)$, we have $q\in f(X)$ and $fx_{n}\to q$ as $n\to\infty$. Now from the above discussion, $q$ is a point of coincidence.

\emph{Uniqueness of point of coincidence:}

Assume that there exist points $p,p^{*}$ in $X$ such that $p=fu=Tu$ and $p^{*}=fu^{*}=Tu^{*}$, for some $u,u^{*}$ in $X.$ Now
\begin{align*}
M(u,u^{*}) &= \max\left\{d(fu,fu^{*}),d(fu,Tu),d(fu^{*},Tu^{*}),\frac{d(fu,Tu^{*})+d(fu^{*},Tu)}{2}\right\} \\
&= \max\left\{d(fu,fu^{*}),d(fu,fu),d(fu^{*},fu^{*}),\frac{d(fu,fu^{*})+d(fu^{*},fu)}{2}\right\} \\
&= d(fu,fu^{*})
\end{align*}
and from the inequality (1.2) we have
\begin{align*}
d(p,p^{*}) &= d(Tu,Tu^{*}) \\
&\leq \delta d(fu,fu^{*})+L\min\left\{d(fu,Tu),d(fu^{*},Tu^{*}),d(fu,Tu^{*}),d(fu^{*},Tu)\right\} \\
&= \delta d(fu,fu^{*})+L\min\left\{d(fu,fu),d(fu^{*},fu^{*}),d(fu,fu^{*})\right\}.
\end{align*}
Thus, it follows that
\begin{align*}
d(p,p^{*}) &\leq \delta d(fu,fu^{*}) \\
&= \delta d(p,p^{*}),
\end{align*}
we deduce that $p=p^{*}$.
\end{proof}

\begin{theorem}\label{thm2.2}
Let $(X,d)$ be a metric space. Let $T,f:X\to X$ be such that $T(X)\subseteq f(X)$. Assume that $T$ satisfies generalized condition $(B)$ associated with $f$. If either $f(X)$ or $T(X)$ is a complete subspace of $X$, then $T$ and $f$ have a unique common fixed point in $X$ provided that the pair $(T,f)$ is weakly compatible.
\end{theorem}

\begin{proof}
By Theorem 2.1, $T$ and $f$ have a unique point of coincidence. Since the pair $(T,f)$ is weakly compatible, by Lemma 1.4 \cite{Abbas2008}, $T$ and $f$ have a unique common fixed point.
\end{proof}

\begin{corollary}[2.3]\label{cor2.3}
Let $(X,d)$ be a metric space. Let $T,f:X\to X$ be such that $T(X)\subseteq f(X)$. Assume that there exist $\delta\in ]0,1[$ and $L\geq 0$ such that
\[d(Tx,Ty)\leq\delta m(x,y)+L\min\{d(fx,Tx),d(fy,Ty),d(fx,Ty),d(fy,Tx)\}\]
for all $x,y\in X$, where
\[m(x,y)=\max\left\{d(fx,fy),\frac{1}{2}[d(fx,Tx)+d(fy,Ty)],\frac{1}{2}[d(fy,Tx)+d(fx,Ty)]\right\}.\]
\end{corollary}
If either $f(X)$ or $T(X)$ is a complete subspace of $X$, then $T$ and $f$ have a point of coincidence. Moreover $T$ and $f$ have a unique common fixed point provided that the pair $(f,T)$ is weakly compatible.

\begin{proof}
As the inequality (2.2) is a special case of (1.2), the result follows from Theorem 2.2 \cite{AlThagafi2006}.
\end{proof}
The following example is in support of Theorem 2.2.

\begin{example}\label{ex2.4}
Let $X=[0,1)$ with usual metric. Define $T,f:X\to X$ as
\[
T(x)=\begin{cases}
\frac{1}{2} & \text{if } 0\leq x<\frac{2}{3}\\
\frac{2}{3} & \text{if } \frac{2}{3}\leq x<1
\end{cases}
\quad\text{and}\quad
f(x)=\begin{cases}
\frac{5}{6} & \text{if } 0\leq x<\frac{2}{3}\\
\frac{4}{3}-x & \text{if } \frac{2}{3}\leq x<1.
\end{cases}
\]

We observe that $T(X)\subset f(X)$ and the pair $(f,T)$ is weakly compatible on $X$. Also, $f$ and $T$ satisfy the inequality (1.2) with $\delta=\frac{1}{2}$ and $L=0$. Hence $f$ and $T$ satisfy all hypotheses of Theorem 2.2 and $\frac{2}{3}$ is the unique common fixed point of $f$ and $T$.

But, when $x\in[0,\frac{2}{3})$ and $y=\frac{2}{3}$, we have $d(Tx,Ty)=\frac{1}{6}$; and $d(fx,fy)=\frac{1}{6}$ so that for any $\alpha\in[0,1)$, $T$ fails to be an $f$-contraction. Hence Theorem 1.4 is not applicable.

This example shows that Theorem 2.2 is a generalization of Theorem 1.4 \cite{AlThagafi2006}.
\end{example}

By choosing $f=I_{X}$ in Theorem 2.2, we have the following corollary.

\begin{corollary}\label{cor2.5}
Let $(X,d)$ be a metric space. Let $T:X\to X$ satisfies generalized condition $(B)$. If $T(X)$ is a complete subspace of $X$, then $T$ has a unique fixed point.
\end{corollary}

\begin{remark}\label{rem2.6}
Theorem 1.5 \cite{Babu2008} follows as a corollary to Corollary 2.5. In fact, Example 1.10 shows that Corollary 2.5 is a generalization of Theorem 1.5 \cite{Babu2008}.
\end{remark}

Now, we have the following result on the continuity in the set of common fixed points. Let $F(T,f)$ denote the set of all common fixed points of $T$ and $f$.

\begin{theorem}\label{thm2.7}
Let $(X,d)$ be a metric space. Assume that $T:X\to X$ satisfies generalized condition $(B)$ associated with a selfmap $f$ on $X$. If $F(T,f)\neq\emptyset$, then $T$ is continuous at $p\in F(T,f)$ whenever $f$ is continuous at $p$.
\end{theorem}

\begin{proof}
Fix $p\in F(T,f)$. Let $(z_{n})$ be any sequence in $X$ converging to $p$. Then by
\end{proof}
Taking $y:=z_{n}$ and $x:=p$ in (1.2), we get
\begin{align*}
d(Tp,Tz_{n}) &\leq \delta M(p,z_{n}) + L\min\{d(fp,Tp),d(fz_{n},Tz_{n}), \\
&\quad d(fp,Tz_{n}),d(fz_{n},Tp)\},\ n=1,2,\cdots
\end{align*}
where
\begin{align*}
M(p,z_{n}) &= \max\left\{d(fp,fz_{n}),d(fp,Tp),d(fz_{n},Tz_{n}),\right. \\
&\quad \left.\frac{d(fp,Tz_{n})+d(fz_{n},Tp)}{2}\right\}
\end{align*}
which, in view of $Tp=fp$, is equivalent to
\begin{align*}
d(Tp,Tz_{n}) &\leq \delta \max\left\{d(Tp,fz_{n}),d(fz_{n},Tz_{n}),\right. \\
&\quad \left.\frac{d(Tp,Tz_{n})+d(fz_{n},Tp)}{2}\right\},
\end{align*}
$n=1,2,\cdots$. Now, by letting $n\to\infty$ we get $Tz_{n}\to Tp$ as $n\to\infty$ whenever $f$ is continuous at $p$ and $0<\delta<1$.

\section{DISCUSSION}

Following the similar arguments to those given in the proof of Theorem 2.2, we can prove the following theorem.

\begin{theorem}\label{thm3.1}
Let $(X,d)$ be a metric space. Let $T,f:X\to X$ be such that $T(X)\subseteq f(X)$. Assume that there exist a constant $\delta\in ]0,\frac{1}{2}[$ and some $L\geq 0$ such that
\[d(Tx,Ty)\leq\delta m(x,y)+L\min\{d(fx,Tx),d(fy,Ty),d(fx,Ty),d(fy,Tx)\}\tag{3.1}\]
for all $x,y\in X$, where
\[m(x,y)=\max\{d(fx,fy),d(fx,Tx),d(fy,Ty),d(fx,Ty),d(fy,Tx)\}.\]
If either $f(X)$ or $T(X)$ is a complete subspace of $X$, then $T$ and $f$ have a unique point of coincidence. Moreover $T$ and $f$ have a unique common fixed point provided that the pair $(f,T)$ is weakly compatible.
\end{theorem}

Now the following question is natural:

\textbf{Open problem 1.} Is Theorem 3.1 valid for $\frac{1}{2}\leq\delta<1$?

If this open problem is solved affirmatively, then Theorem 3.1 together with the solution of open problem 1 extends Theorem 1.11 (Theorem 3.4 of Berinde \cite{Berinde2008}) to a pair of selfmaps.

Berinde \cite{Berinde2008} introduced the concept of \v{C}iri\'c almost contraction, that is, a mapping for which there exist a constant $\alpha\in[0,1[$ and some $L\geq 0$ such that
\[d(Tx,Ty)\leq\alpha M(x,y)+Ld(y,Tx),\text{ for all }x,y\in X,\]
where $M(x,y)=\max\{d(x,y),d(x,Tx),d(y,Ty),d(x,Ty),d(y,Tx)\}$.

Berinde proved the following two fixed point theorems for this class of mappings in complete metric spaces.

\begin{theorem}[(Berinde \cite{Berinde2008}, Theorem 3.2)]\label{thm3.2}
Let $(X,d)$ be a complete metric space and $T:X\to X$ a \v{C}iri\'c almost contraction. Then
\begin{enumerate}
\item[(1)] $F(T)=\{x\in X:Tx=x\}\neq\emptyset$,
\item[(2)] for any $x_{0}\in X$, the Picard iteration $\{x_{n}\}_{n=1}^{\infty}$ defined by (1.1) converges to some $x^{*}\in F(T)$,
\item[(3)] the following estimate 
\[d(x_{n},x^{*})\leq\frac{\alpha^{n}}{(1-\alpha)^{2}}d(x_{0},x_{1})\]
holds, for $n=1,2,\cdots$.
\end{enumerate}
\end{theorem}

\begin{theorem}[(Berinde \cite{Berinde2008}, Theorem 3.3)]\label{thm3.3}
Let $(X,d)$ be a complete metric space and $T:X\to X$ a \v{C}iri\'c almost contraction. If there exist $\theta\in ]0,1[$ and some $L_{1}\geq 0$ such that
\[d(Tx,Ty)\leq\theta d(x,y)+L_{1}d(x,Tx),\quad\text{ for all }x,y\in X.\]
Then
\begin{enumerate}
\item[(1)] $T$ has a unique fixed point, i.e., $F(T)=\{x^{*}\}$,
\item[(2)] for any $x_{0}\in X$, the Picard iteration $\{x_{n}\}_{n=1}^{\infty}$ defined by (1.1) converges to some $x^{*}\in F(T)$,
\item[(3)] the a priori error estimate (3) of Theorem 3.2 holds,
\item[(4)] the rate of convergence of Picard iteration is given by 
\[d(x_{n},x^{*})\leq\theta d(x_{n-1},x^{*})\]
for $n=0,1,2,\cdots$.
\end{enumerate}
\end{theorem}

But the following example shows that for a pair of selfmaps $T$ and $f$ of a complete metric space $X$, even if $T$ is an almost contraction with respect to $f$ and both $f$ and $T$ are continuous on $X$, the maps $T$ and $f$ may not have a common fixed point.

\begin{example}\label{ex3.4}
Let $X=\mathbb{R}$, the real line with the usual metric. We define mappings $f,T:X\to X$ by $Tx=\frac{x+1}{4}$ and $fx=\frac{x}{2},\ x\in X$.

Then, with $\delta=\frac{1}{2}$ and for any $L\geq 0$, $T$ is an almost contraction with respect to $f$. But $f$ and $T$ have no common fixed points.
\end{example}

Thus the following question is possible:

\textbf{Open problem 2.} Under what additional assumptions, \emph{either} on $T$ and $f$ \emph{or} on the domain of $T$ and $f$, the maps $T$ and $f$ have common fixed points?

\end{document}